\documentclass[12pt]{amsart}
\usepackage{geometry} 
\newcounter{dummy} \numberwithin{dummy}{section}
\newtheorem{theorem}[dummy]{Theorem}
\newtheorem{definition}[dummy]{Definition}
\newtheorem{proposition}[dummy]{Proposition}
\newtheorem{lemma}[dummy]{Lemma}
\newtheorem{remark}[dummy]{Remark}
\newtheorem{example}[dummy]{Example}

\usepackage{graphicx}
\usepackage{amssymb}
\usepackage{epstopdf}
\usepackage{lscape}
\usepackage{indentfirst}
\usepackage{latexsym}
\usepackage{multirow}
\usepackage{tabls}
\usepackage{amsthm}
\usepackage{amsthm,amsmath}
\usepackage{graphicx}
\usepackage{indentfirst}
\usepackage{latexsym}
\usepackage{multirow}
\usepackage{tabls}
\usepackage{wrapfig}
\usepackage[all]{xy}
\usepackage[numbers]{natbib}
\usepackage{tikz}
\usepackage{mathrsfs}
\newtheorem{corollary}[dummy]{Corollary}
\usepackage{mathtools}

\usepackage[OT2,T1]{fontenc}
\DeclareSymbolFont{cyrletters}{OT2}{wncyr}{m}{n}
\DeclareMathSymbol{\Sha}{\mathalpha}{cyrletters}{"58}

\usepackage{tikz-cd}

\usepackage{float}
\counterwithin{table}{section}
\usepackage{hyperref}
\usepackage{enumitem}

\newcommand{\Eleni}[1]{\textcolor{blue}{{\sf (Eleni:} {{#1})}}}

\title[Elliptic Curves with positive rank and no integral points]{Elliptic Curves with positive rank \\ and no integral points}
\author{eleni agathocleous }

\begin{document}
\begin{center}
\maketitle{\textbf{Abstract \footnote[1]{This work has been partially supported by the European Union’s H2020 Programme under grant agreement number ERC-669891, and partially supported with funds from the Ministry of Science, Research and Culture of the State of Brandenburg within the Centre for Quantum Technology and Applications (CQTA).}}}
\end{center}

\tiny
We consider all \emph{odd} fundamental discriminants $D \equiv 2 \bmod 3$ and their mirror discriminants $D' = -3D$, and we study the family of elliptic curves $E_{D'}: y^{2} = x^{3} + 16D'$. We denote by $r_{3}(D)$ and $r_{3}(D')$ the rank of the $3$-part of the ideal class group of $\mathbb{Q}(\sqrt{D})$ and $\mathbb{Q}(\sqrt{D'})$ respectively. We show that every curve in the subfamily of elliptic curves $E_{D'}$ with $r_{3}(D) = r_{3}(D') + 1$ for $D < 0$ (respectively, with $r_{3}(D) = r_{3}(D')$ for $D > 0$) cannot have any integral points, and this is proved unconditionally. By employing results of Satg\'e and by assuming finiteness of the $3$-primary part of their Tate-Shafarevich group, we show that the curves $E_{D'}$ must have odd rank when $D < 0$ and even rank when $D > 0$. This result is particularly interesting for the case of $D < 0$ since every curve $E_{D'}$ with $r_{3}(D) = r_{3}(D') + 1$ has infinitely many rational points - assuming finiteness of the $3$-primary part of their Tate-Shafarevich group - yet no integral points. We obtain an unconditional result on the existence of elliptic curves with non-trivial rank and no integral points, by defining a parametrised family of such curves with no integral points but with a parametrised rational point, which we prove that it is of infinite order.

\


\


\ 

\begin{center}
Eleni Agathocleous\\ Deutsches Elektronen-Synchrotron DESY, \\ Platanenallee 6, 15738 Zeuthen, Germany \\ email: eleni.agathocleous@desy.de 
\end{center}

\normalsize

\section{Introduction}\label{intro} 
As it is well known, if an elliptic curve has non-trivial rank, then it has infinitely many rational points. However, as it was proved by Siegel \cite{Sieg}, an elliptic curve can have only finitely many integral points. Siegel's theorem is not effective and given an elliptic curve, it is in general hard to determine how many integral points this curve might have, or if it has any integral points at all. Other techniques that have been developed over the years in order to study this problem, concentrate on finding bounds for the number of these integral points. Some of these bounds depend for example on the rank of the elliptic curve, the number of distinct prime divisors of the discriminant, or the primes of bad multiplicative reduction (e.g. \cite{BhaEtAl}, \cite{HelVen}, \cite{Sil} and references therein). 

In this paper we study the family of $j$-invariant zero elliptic curves $E_{D'}: y^{2} = x^{3} + 16D'$, where $D'=-3D$ and $D$ is any \emph{odd} fundamental discriminant equivalent to $2$ modulo $3$. The basic definitions and important properties of these curves are given in Section~\ref{DefPr}. 

In Section~\ref{SelmerRank}, by employing results of Satg\'e \cite{Satge} and by assuming finiteness of the $3$-primary part of their Tate-Shafarevich group, we obtain the following conditional result on the parity of the rank of the elliptic curves $E_{D'}$:

\begin{proposition}
\label{proposition1}
For every odd fundamental discrimant $D \equiv 2 \bmod 3$ denote by $D'$ its mirror discriminant
$D'=-3D$ and let
\[
E_{D'} : y^2 = x^3 + 16D'
\]
be the corresponding elliptic curve of $j$-invariant zero. Assume that the $3$-primary part of
the Tate--Shafarevich group is finite. Then
\[
\operatorname{rank} E_{D'}(\mathbb Q) \equiv
\begin{cases}
1 \bmod 2 &\text{if $D < 0$,} \\
0 \bmod 2 &\text{if $D > 0$.} \\
\end{cases}
\]
\end{proposition}

In Section~\ref{Escalatory} we use the connection between the points of $E_{D'}(\mathbb{Q})$ and the so-called $3$-virtual units of $\mathbb{Q}(\sqrt{D'})$, via the Fundamental $3$-Descent Map, and we prove the following \emph{unconditional} result regarding the set of integral points $E_{D'}(\mathbb{Z})$. 
\begin{proposition}
\label{proposition2}
For every odd fundamental discrimant $D \equiv 2 \bmod 3$ denote by $D'=-3D$ its mirror discriminant and let
\[
E_{D'} : y^2 = x^3 + 16D'
\]
be the corresponding elliptic curve of $j$-invariant zero. Let $r_{3}(D)$ and $r_{3}(D')$ denote the rank of the $3$-part of the ideal class group of $\mathbb{Q}(\sqrt{D})$ and $\mathbb{Q}(\sqrt{D'})$ respectively. Assume that 
\[\begin{cases}
r_{3}(D)=r_{3}(D') + 1, \ \text{if} \ D < 0 \\
r_{3}(D)=r_{3}(D'), \ \text{if} \ D > 0. 
\end{cases}\]
Then
\[
E_{D'}(\mathbb{Z}) = \emptyset.
\]
\end{proposition}

The results from Propositions \ref{proposition1} and  \ref{proposition2} combined, yield a particularly interesting family of elliptic curves $E_{D'}$ in the case of negative discriminants $D<0$ since, whenever $r_{3}(D)=r_{3}(D') + 1$, they imply that the curves $E_{D'}$ have infinitely many rational points but no integral points. This result is conditional, relying on the assumption of finiteness of the $3$-primary part of the Tate-Shafarevich group. We obtain an unconditional result on the existence of such curves in Section~\ref{Example}. 

Particularly, in Section~\ref{Example} we define the parametrised negative integer 
$$D(w) = -(2^{4}3^{7}w^{2} + (2+2^{4}3^{7})w + 2^{2}3^{7}+1)$$ 
which yields the rational point 
$$P(w) = (1/9 , (2^{4}3^{7}w+2^{3}3^{7}+1)/27) \in E_{D(w)'}(\mathbb{Q}), \ \text{for any integer}\ w.$$ 
Now, for $w \equiv 9 \bmod 12$ these integers always satisfy $D(w) \equiv 2 \bmod 3$, and we prove the following lemma:
\begin{lemma}
With notation as above, the following set of odd negative fundamental discriminants is infinite:
$$\mathcal{D}_{w} = \{ D(w) \ | \ w \equiv 9 \bmod 12 \ \text{and} \ D(w) \ \text{squarefree} \}.$$   
\end{lemma}
Given any such discriminant $D(w) \in \mathcal{D}(w)$, we prove that the corresponding curve $E_{D(w)'}$ has positive rank, by showing the following:
\begin{proposition}\label{prop:InfOrder}
For every $D(w) \in \mathcal{D}(w)$, consider the point $$P(w) = (1/9 , y(w)/27), \ \text{with} \ y(w) = 2^{4}3^{7}w+2^{3}3^{7}+1.$$
The point $P(w)$ belongs to the set $E_{D(w)'}(\mathbb{Q}) \setminus \hat{\phi}(E_{D(w)}(\mathbb{Q}))$, and therefore the corresponding elliptic curve $E_{D(w)'}/\mathbb{Q}$ has positive rank.
\end{proposition}
It then follows that every discriminant $D(w)$ from the set $\mathcal{D}_{w}^{esc}$ below
$$\mathcal{D}_{w}^{esc} \coloneqq \{ D(w) \in \mathcal{D}_{w} \ | \ r_{3}(D(w)) = r_{3}(D(w)') + 1 \}$$ 
corresponds to an elliptic curve $E_{D(w)'}$ that has infinitely many rational points, yet no integral points. This result is unconditional and the set $\mathcal{D}_{w}^{esc}$ is of course not empty. A straightforward search in PARI/GP for $1 \leq w \leq 10^{5}$ with $w \equiv 9 \bmod 12$ and $D(w)$ squarefree yielded $4,842$ such discriminants $D(w)$ that lie within the bounds
$$D(9) = -3158047 \leq D(w) \leq D(99993) = -349874512078399.$$

The case of imaginary quadratic number fields $\mathbb{Q}(\sqrt{D})$ with $r_{3}(D) = r_{3}(D')+1$, known as \emph{escalatory} case, is a consequence of a property of the units $E$ of the field $L=\mathbb{Q}(\sqrt{D},\sqrt{D'})$. More specifically, if the orthogonal idempotent associated with the norm from $L$ down to $\mathbb{Q}(\sqrt{D'})$, evaluated at $E/E^{3}$ does not vanish modulo cube powers of the units of $\mathbb{Q}(\sqrt{D'})$, then it `adds one' to the $3$-rank of $Cl(\mathbb{Q}(\sqrt{D}))$ - hence the escalatory case. For a proof of this, the interested reader may refer to \cite[Theorem 10.10]{Wash}.

\section{Definitions and Preliminaries} \label{DefPr}
We denote by $D$ all \emph{odd} fundamental discriminants with $D \equiv 2 \bmod 3$, and by $D' = -3D$ their mirror discriminants, which are also odd and fundamental. Hence, the odd squarefree integers $D$ and $D'$ satisfy the following congruence conditions
\begin{equation} \label{conditionsD} D, D' \equiv 1 \bmod 4 \ \text{and} \ D \equiv 2 \bmod 3,\end{equation}
yielding further that $$D \equiv 5 \bmod 12 \ \text{and} \ D' \equiv 9 \bmod 12.$$

We denote by $K_{D}$ and $K_{D'}$ the quadratic number fields $K_{D} = \mathbb{Q}(\sqrt{D})$ and $K_{D'} = \mathbb{Q}(\sqrt{D'})$. We call the field $K_{D'}$, the quadratic resolvent of $K_{D}$. 

For any number field $M/\mathbb{Q}$ we denote by $\overline{M}$ its algebraic closure, and by $E(M)$ the group of points of any elliptic curve $E$ defined over $M$. For the theory and proofs of the facts that we present below, the interested reader may refer for example to \cite[Chapter IV]{CassFr}, \cite{Satge}, \cite[Chapter X.4 and Appendix B]{Silverman}, and \cite{Top}. 
  
We consider the family of $j$-invariant zero elliptic curves
\begin{equation} E_{D'}: y^2 = x^3 + 16D'. \end{equation}
The torsion group $$\mathcal{T}_{D'} = \{ O, (0,\pm 4 \sqrt{D'}) \} \subseteq E_{D'}(\mathbb{\overline{Q}})$$ 
is a subgroup of $E_{D'}(\mathbb{\overline{Q}})$ of order $3$ and is invariant under the action of $G_{\mathbb{Q}} = Gal(\overline{\mathbb{Q}}/\mathbb{Q})$. The quotient of $E_{D'}$ over $\mathcal{T}_{D'}$ defines a family of isogenous elliptic curves given by the equation
\begin{equation} E_{D}: Y^{2} = X^{3} + 16\cdot3^{4}D . \end{equation}
The $3$-torsion group for the curves $E_{D}$, invariant under the action of $G_{\mathbb{Q}}$, is
$$\mathcal{T}_{D} =   \{ O,(0,\pm 36 \sqrt{D}) \} \subseteq E_{D}(\overline{\mathbb{Q}}).$$ 
Let $\phi$ denote the rational $3$-isogeny $$\phi : E_{D'} \rightarrow E_{D} $$ with kernel $\text{ker}_{\phi}(\mathbb{\overline{Q}}) = \mathcal{T}_{D'}$. The isogeny $\phi$ is given by the rational maps (\cite[Proposition 8.4.3]{Cohen})
\begin{equation} \label{themap} \phi((x , y)) = \Big{(} \frac{x^{3} +2^{6}D'}{x^{2}} , \frac{y(x^3 - 2^{7}D')}{x^{3}} \Big{)}.\end{equation}
The dual isogeny 
$$\hat{\phi}: E_{D} \rightarrow E_{D'}$$ 
has kernel $\text{ker}_{\hat{\phi}}(\mathbb{\overline{Q}}) = \mathcal{T}_{D}$ and is given by the rational maps (\cite[Proposition 8.4.3]{Cohen})
\begin{equation} \label{themapdual} \hat{\phi}((X , Y)) = \Big{(} \frac{X^{3} - 2^{6}3^{3}D'}{9X^{2}} , \frac{Y(X^3 + 2^{7}3^{3}D')}{27X^{3}} \Big{)}.\end{equation}
The isogenies $\phi$ and $\hat{\phi}$ satisfy $\phi\circ\hat{\phi} = [3]_{D}$ and $\hat{\phi}\circ\phi = [3]_{D'}$, where $[3]_{D}$ and $[3]_{D'}$ are the multiplication-by-$3$ maps on $E_{D}$ and $E_{D'}$ respectively. 

Consider the exact sequence (\cite[Section X.4, Remark 4.7]{Silverman}) 
\begin{equation} \label{splitrank} 0 \rightarrow E_{D}(\mathbb{Q})[\hat{\phi}]/\phi(E_{D'}(\mathbb{Q})[3]) \rightarrow E_{D}(\mathbb{Q})/\phi(E_{D'}(\mathbb{Q})) \xrightarrow{\hat{\phi}} \end{equation} 
$$\xrightarrow{\hat{\phi}} E_{D'}(\mathbb{Q})/3E_{D'}(\mathbb{Q}) \rightarrow E_{D'}(\mathbb{Q})/\hat{\phi}(E_{D}(\mathbb{Q})) \rightarrow 0 .$$
Since both $\text{ker}_{\phi}(\mathbb{\overline{Q}}) = \mathcal{T}_{D'}$ and $\text{ker}_{\hat{\phi}}(\mathbb{\overline{Q}}) = \mathcal{T}_{D}$ contain no non-trivial rational point of order $3$, the first quotient group of (\ref{splitrank}) vanishes and the rank $r(E_{D'})$ of $E_{D'}(\mathbb{Q})$ equals
\begin{equation}\label{R} \begin{gathered} r(E_{D'}) = \text{dim}_{\mathbb{F}_{3}}(E_{D'}(\mathbb{Q})/3E_{D'}(\mathbb{Q})) =\\  \text{dim}_{\mathbb{F}_{3}}(E_{D}(\mathbb{Q})/\phi(E_{D'}(\mathbb{Q}))) + \text{dim}_{\mathbb{F}_{3}}(E_{D'}(\mathbb{Q})/\hat{\phi}(E_{D}(\mathbb{Q})) ).\end{gathered} \end{equation} 

Consider now the short exact sequence \begin{equation}\label{lisogeny} 0 \rightarrow \mathcal{T}_{D'} \rightarrow E_{D'}(\mathbb{\overline{Q}}) \xrightarrow{\phi} E_{D}(\mathbb{\overline{Q}}) \rightarrow 0.\end{equation} From this, we obtain the long exact cohomology sequence which gives in particular the following 
\begin{equation}\label{quotient} 0 \rightarrow E_{D}(\mathbb{Q})/\phi(E_{D'}(\mathbb{Q}))  \xrightarrow{\delta}  H^{1}(G_{\mathbb{Q}},\mathcal{T}_{D'}) \rightarrow H^{1}(G_{\mathbb{Q}}, E_{D'}(\mathbb{\overline{Q}}))[\phi] \rightarrow 0.\end{equation} 
By localising at each prime $p$, we obtain the following commutative diagram, where $res_{p}$ is the usual restriction map: 
\small
\begin{center}
\begin{tikzcd}[ampersand replacement=\&, column sep=tiny]
        0 \arrow[r] \&  E_{D}(\mathbb{Q})/\phi(E_{D'}(\mathbb{Q})) \arrow[r, "\delta"] \arrow[d]  \& H^{1}(G_{\mathbb{Q}},\mathcal{T}_{D'}) \arrow[r]  \arrow[d, "\underset{p}{\prod}res_{p}"] \& H^{1}(G_{\mathbb{Q}}, E_{D'}(\mathbb{\overline{Q}}))[\phi]  \arrow[r] \arrow[d, "\underset{p}{\prod}res_{p}"]  \& 0 \\  
        0 \arrow[r] \&  \underset{p}{\prod}E_{D}(\mathbb{Q}_{p})/\phi(E_{D'}(\mathbb{Q}_{p})) \arrow[r]  \& \underset{p}{\prod}H^{1}(G_{\mathbb{Q}_{p}},\mathcal{T}_{D'}) \arrow[r]  \& \underset{p}{\prod}H^{1}(G_{\mathbb{Q}_{p}}, E_{D'}(\mathbb{\overline{Q}}_{p}))[\phi]  \arrow[r]  \& 0 
    \end{tikzcd}
\end{center}    
\normalsize  

\begin{definition}\label{definition} The Selmer group of $E_{D'}$ relative to the isogeny $\phi$ is $$\mathcal{S}_{\phi}(E_{D'}) = \{ x \in H^{1}(G_{\mathbb{Q}},\mathcal{T}_{D'}) \ | \ res_{p}(x) \in \emph{Im}( E_{D}(\mathbb{Q}_{p})/\phi(E_{D'}(\mathbb{Q}_{p}))) \ \text{for all p}\}.$$

The Tate-Shafarevich group of $E_{D'}$ can now be defined as $$\Sha(E_{D'}) = \{x \in H^{1}(G_{\mathbb{Q}}, E_{D'}(\mathbb{\overline{Q}})) \  | \ res_{p}(x) = 0 \ \text{for all} \ p\}.$$

These two groups are connected together as follows: 
\begin{equation} \label{SelSha} 0 \rightarrow E_{D}(\mathbb{Q})/\phi(E_{D'}(\mathbb{Q})) \rightarrow \mathcal{S}_{\phi}(E_{D'}) \rightarrow \Sha(E_{D'})[\phi] \rightarrow 0.\end{equation} \end{definition}

\begin{remark}\label{dualisogenydefn} By considering the dual isogeny $\hat{\phi}$ instead, we obtain exact sequences analogous to (\ref{lisogeny}), (\ref{quotient}) and (\ref{SelSha}), which in turn give us the analogous definitions for $\mathcal{S}_{\hat{\phi}}(E_{D})$,  
$\Sha(E_{D})$ and $\Sha(E_{D})[\hat{\phi}]$. \end{remark}

\begin{remark}\label{SelSha3} We also obtain exact sequences analogous to (\ref{lisogeny}), (\ref{quotient}) and (\ref{SelSha}) by considering the map $[3]_{D} = \phi\circ\hat{\phi}$ (respectively $[3]_{D'} = \hat{\phi}\circ\phi$), which in turn give us the analogous definitions for $\mathcal{S}_{3}(E_{D'})$,  $\Sha(E_{D'})$ and $\Sha(E_{D'})[3]$ (respectively $\mathcal{S}_{3}(E_{D'})$,  $\Sha(E_{D'})$ and $\Sha(E_{D'})[3]$). \end{remark}


\section{On the $3$-Selmer group and rank of the elliptic curves $E_{D'}$} \label{SelmerRank} 
As above, we let $D$ be any odd fundamental discriminant $D \equiv 2 \bmod 3$, and we let $r(\mathcal{S}_{\phi}(E_{D'}))$ and $r(\mathcal{S}_{\hat{\phi}}(E_{D}))$ denote the rank, as $\mathbb{F}_{3}$-vector spaces, of the Selmer groups $\mathcal{S}_{\phi}(E_{D'})$ and $\mathcal{S}_{\hat{\phi}}(E_{D})$, relative to the isogenies $\phi$ and $\hat{\phi}$. We denote by $r_{3}(D)$ and $r_{3}(D')$ the rank of the $3$-part of the ideal class group $Cl(K_{D})$ and $Cl(K_{D'})$ of $K_{D}$ and $K_{D'}$ respectively. By employing the results of Satg\'e \cite[Section 3]{Satge}, we compute below the precise rank for the Selmer groups $\mathcal{S}_{\phi}$ and $\mathcal{S}_{\hat{\phi}}$, in terms of the ranks $r_{3}(D)$ and $r_{3}(D')$.

\begin{proposition}\label{Selmer} For all \emph{odd} fundamental discriminants $D \equiv 2 \bmod 3$, with mirror discriminant $D' = -3D$, the rank of the Selmer groups $\mathcal{S}_{\phi}(E_{D'})$ and $\mathcal{S}_{\hat{\phi}}(E_{D})$ of the curves $E_{D'}$ and $E_{D}$ are as follows: 
\begin{equation*} r(\mathcal{S}_{\phi}(E_{D'})) = r_{3}(D') \ \text{and} \ r(\mathcal{S}_{\hat{\phi}}(E_{D})) = r_{3}(D') + 1, \ \text{for} \ D < -4\end{equation*}  
\begin{equation*} r(\mathcal{S}_{\phi}(E_{D'})) = r(\mathcal{S}_{\hat{\phi}}(E_{D})) = r_{3}(D'), \ \text{for} \ D > 4.\end{equation*} 
\end{proposition}
\begin{proof} Our elliptic curves $E_{D'}$ have a constant term equal to $16D'$. With $D'$ squarefree and with $2^4 || 16D'$, Lemma 3.1 in \cite{Satge} is vacuously true. Now $3 || 16D'$ and, given the congruence condition $D \equiv 2 \bmod 3$, we have that $-16D \equiv 1 \bmod 3$. Therefore, from Proposition 3.2(1) of \cite{Satge} we have that $r(\mathcal{S}_{\phi}(E_{D'})) = r_{3}(D')$. Finally, Proposition 3.3.(1) of \cite{Satge} gives $r(\mathcal{S}_{\hat{\phi}}(E_{D})) = r_{3}(D') + 1$ when $16D' > 0$, i.e. when $D$ is negative, and $r(\mathcal{S}_{\hat{\phi}}(E_{D})) = r_{3}(D')$ when $16D' < 0$, i.e. when $D$ is positive. \end{proof}

As in Remark~\ref{SelSha3}, we denote by $\mathcal{S}_{3}(E_{D})$ and $\mathcal{S}_{3}(E_{D'})$ the $3$-Selmer group of the corresponding elliptic curves. Their rank will be denoted by $r(\mathcal{S}_{3}(E_{D}))$ and $r(\mathcal{S}_{3}(E_{D'}))$ respectively. We now consider the exact sequence (\cite[Corollary 1]{Kloo})
\footnotesize \begin{equation} \label{SelmerSha} 0 \rightarrow \frac{E_{D}(\mathbb{Q})[\hat{\phi}]}{\phi(E_{D'}(\mathbb{Q})[3])} \rightarrow \mathcal{S}_{\phi}(E_{D'}) \rightarrow  \mathcal{S}_{3}(E_{D'}) \rightarrow \mathcal{S}_{\hat{\phi}}(E_{D}) \rightarrow \frac{\Sha(E_{D})[\hat{\phi}]}{\phi(\Sha(E_{D'})[3])} \rightarrow 0 . \end{equation} \normalsize
Since our curves have no rational $3$-torsion points, the first term of (\ref{SelmerSha}) is trivial. On the other hand,  the non-degenerate alternating pairing defined by Cassels in \cite{Cassels}, induces a non-degenerate alternating pairing on the $\mathbb{F}_{3}$-vector space $\frac{\Sha(E_{D})[\hat{\phi}]}{\phi(\Sha(E_{D'})[3])}$, which implies that it is even-dimensional \cite[Proposition 42]{Bhar2}. Therefore, we obtain the following result regarding the parity of the rank of the $3$-Selmer group and the ranks of the two Selmer groups $\mathcal{S}_\phi(E_{D'})$ and $\mathcal{S}_{\hat{\phi}}(E_{D})$:
\begin{equation}\label{parity} r(\mathcal{S}_{3}(E_{D'})) \equiv r(\mathcal{S}_{\phi}(E_{D'})) + r(\mathcal{S}_{\hat{\phi}}(E_{D})) \bmod 2. \end{equation}

\begin{corollary} \label{oddrank} 
We denote by $D$ every odd fundamental discriminant with $D \equiv 2 \bmod 3$, and we let
$D'=-3D$ and
\[
E_{D'} : y^2 = x^3 + 16D'
\]
be the corresponding elliptic curve with $j$-invariant zero. We assume that the $3$-primary part of the Tate-Shafarevich group is finite. Then
\[
r(E_{D'}) \equiv
\begin{cases}
1 \bmod 2 &\text{if $D < 0$,} \\
0 \bmod 2 &\text{if $D > 0$.} \\
\end{cases}
\]
\end{corollary}
\begin{proof} Given Remark~\ref{SelSha3}, the exact sequence analogous to (\ref{SelSha}) is
\begin{equation}\label{eqnSelSha3} 0 \rightarrow E_{D}(\mathbb{Q})/3(E_{D'}(\mathbb{Q})) \rightarrow \mathcal{S}_{3}(E_{D'}(\mathbb{Q})) \rightarrow \Sha(E_{D'}(\mathbb{Q}))[3] \rightarrow 0. \end{equation}
By assuming finiteness of the $3$-primary part of their Tate-Shafarevich group, the non-degenerate alternating pairing on $\Sha(E_{D'})[3]$ defined by Cassels in \cite{Cassels} implies that this group is of even dimension as an $\mathbb{F}_{3}$-vector space, and we thus obtain the congruence relation
\begin{equation}\label{ParityRank} r(E_{D'}) \equiv r(\mathcal{S}_{3}(E_{D'})) \equiv r(\mathcal{S}_{\phi}(E_{D'})) + r(\mathcal{S}_{\hat{\phi}}(E_{D})) \equiv \begin{cases}
1 \bmod 2, \ \text{for} \ D < 0\\
0 \bmod 2,\ \text{for} \ D > 0
\end{cases}.
\end{equation} 
The last two equivalences follow from (\ref{parity}) and Proposition~\ref{Selmer}.
\end{proof}


\section{The Case of No Integral Points }\label{Escalatory}
In this section we will define a subfamily of the curves $E_{D'}$ and show in Proposition~\ref{NoIntegral} that this subfamily cannot have any integral points. This is particularly interesting for the case of $D < 0$ since, by Corollary~\ref{oddrank} (assuming finiteness of the $3$-primary part of the Tate-Shafarevich group), every such curve $E_{D'}$ will have infinitely many rational points but no integral points. An unconditional result for the existence of such curves with non-trivial rank and no integral points is given in Section~\ref{Example}, where we construct a parametrised odd fundamental discriminant $D(w) \equiv 2 \bmod 3$ that guarantees the existence of a rational point of infinite order.     

As we will need to employ known results on cubic extensions of quadratic number fields, we recall below some known facts and definitions for completeness. The interested reader may refer to \cite[Chapter 5]{Cohen1} and \cite[Chapter 4]{Hambleton} for more details, and to \cite{Aga} for a short account.

Let $M$ be any number field with ring of integers $O_{M}$. For any prime $\ell$, Cohen in \cite[Chapter 5]{Cohen1} defines the group $V_{\ell}(M)$ of $\ell$-virtual units of $M$, which give rise to extensions of $M$ of degree $\ell$. The group $V_{\ell}(M)$ is directly related to the known $\ell$-Selmer group $Sel_{\ell}(M)$ of $M$. We cite below \cite[Proposition 5.2.3 and Definition 5.2.4]{Cohen1}, adjusted to our needs for $\ell = 3$.

\begin{proposition} \emph{(\cite[Proposition 5.2.3]{Cohen1})} \label{fact1} Let $\mu \in M^{\times}$. The following two properties are equivalent:

(1) There exists an ideal $\mathfrak{a}$ such that $\mu O_{M} = \mathfrak{a}^3$. 

(2) The element $\mu$ belongs to the group generated by the units, the cube powers of elements of $M^{\times}$ and the cube powers of nonprincipal ideals which become principal when raised to the cube power, hence they belong to the $3$-torsion part $Cl(M)[3]$ of the ideal class group $Cl(M)$ of $M$. \end{proposition}

\begin{definition}\label{def3unit}\emph{(\cite[Definition 5.2.4]{Cohen1})} \\ (1) An element $\mu \in M^{\times}$ satisfying one of the two equivalent conditions of the above proposition will be called a $3$-virtual unit. \\ (2) The set of $3$-virtual units forms a multiplicative group which we denote by 
$V_{3}(M)$. \\ (3) The quotient group $Sel_{3}(M) \coloneqq V_{3}(M)/{M^{\times}}^{3}$ will be called the $3$-Selmer group of $M$. \end{definition}
The following short exact sequence \cite[Section 5.2]{Cohen} shows the relation between the $3$-Selmer group $Sel_{3}(M) = V_{3}(M)/{M^{\times}}^{3}$ of $M$ and the $3$-part of its ideal class group and its group of units: 
\begin{equation}\label{Selmershort} 1 \rightarrow \frac{U_{M}}{{U_{M}}^{3}} \rightarrow Sel_{3}(M) \xrightarrow{\Phi} Cl(M)[3] \rightarrow 1. \end{equation}
We notice that the pre-image of $\Phi$ is unique, up to a unit. Of course, in the case that $M$ is an imaginary quadratic number field of fundamental discriminant $\Delta < -4$, there are no units other than $\pm 1$ and $\Phi$ is an isomorphism.

\begin{definition} \label{Primitive} Given an order $\mathcal{O}$, we say that a proper $\mathcal{O}$-ideal $\mathfrak{a}$ is primitive if it is not of the form $k\mathfrak{a}$ for $1 < k \in \mathbb{Z}$ and $\mathfrak{a}$ a proper ideal (\cite[\S11.D, pg.214]{Cox}). We will call an element 
$\pi$ primitive if $(\pi)$ is a primitive ideal.\end{definition} 

It is known that every cubic field $L$ of fundamental discriminant $\Delta$ arises from a $3$-virtual unit in the quadratic resolvent $K_{\Delta'} = \mathbb{Q}(\sqrt{\Delta'})$ of the quadratic number field $K_{\Delta} = \mathbb{Q}(\sqrt{\Delta}) \subseteq L(\sqrt{\Delta})$, with $Gal(L(\sqrt{\Delta})/\mathbb{Q})$ isomorphic to the symmetric group on $3$ elements. The discriminant $\Delta' \coloneqq -3\Delta/\text{gcd}(3,\Delta)^{2}$ is also a fundamental discriminant. As it is proved for example in \cite[Lemma 4.2]{Hambleton}, every such cubic field $L$ is generated by an irreducible cubic polynomial of the form
$$f_{\mu}(x) = x^{3} - 3(\mu \overline{\mu})^{1/3}x + (\mu+\overline{\mu}) \in \mathbb{Z}[x],$$ 
where $\mu \in \mathcal{O}_{\Delta'} \setminus \mathcal{O}_{\Delta'}^{3}$ is a primitive $3$-virtual unit, with $\mathcal{O}_{\Delta'}$ being the maximal order of $K_{\Delta'}$. Two $3$-virtual units $\mu_{1} = \mathfrak{a}_{1}^{3}$ and $\mu_{2}= \mathfrak{a}_{2}^{3}$ are generators of the same field, up to conjugation, if and only if $\mu_{1}$ or $\overline{\mu_{1}}$ is equal to $\alpha^{3} \mu_{2}$, for some $\alpha \in K_{\Delta'}^{\times}$ \cite[Theorem 4.3]{Hambleton}. Finally, we say that the polynomial $f_{\mu}$ is in \emph{standard form} if there is no integer $c \neq \pm 1$ such that $c^{2} | 3(\mu \overline{\mu})^{1/3}$ and $c^{3} | (\mu+\overline{\mu})$.

A classical result known as Scholz's Reflection Theorem, gives us the relation between the $3$-ranks of the ideal class groups of the quadratic number fields $K_{\Delta}$ and $K_{\Delta'}$ (the interested reader may refer to \cite[Section 10.2]{Wash} for more details):
\begin{theorem}\label{Scholz} \textbf{Reflection Theorem of Scholz} \\
Let $\Delta$ and $\Delta' \coloneqq -3\Delta/\text{gcd}(3,\Delta)^{2}$ be fundamental discriminants. Let $r_{3}(real)$ denote the $3$-rank of the ideal class group of the real quadratic number field of discriminant $\Delta$ (or $\Delta'$), and let $r_{3}(imaginary)$ denote the $3$-rank of the ideal class group of the imaginary quadratic number field of discriminant $\Delta'$ (or $\Delta$). Then the following holds true
\begin{equation*}r_{3}(real) \leq r_{3}(imaginary) \leq r_{3}(real) + 1. \end{equation*} \hfill $\square$  \end{theorem}

\begin{definition} With notation as in Theorem~\ref{Scholz}, we define as \emph{escalatory} the case where $r_{3}(imaginary) = r_{3}(real) + 1$ and as \emph{non-escalatory} the case where $r_{3}(imaginary) = r_{3}(real)$. \end{definition}

The terms \emph{escalatory} and \emph{non-escalatory} are used for example in \cite[Chapter 4]{Hambleton}. Specifically, in Section 4.10 of \cite{Hambleton}, it is shown that in the case of negative fundamental discriminants $\Delta < 0$, the escalatory case is equivalent to the non-existence of cubic fields of discriminant $3^{4}\Delta$. The opposite is true for $\Delta > 0$ where the non-existence of cubic fields of discriminant $3^{4}\Delta$ is equivalent to the non-escalatory case. Translating these facts to our notation with the discriminants $D$ and $D' = -3D$ as above, we have 
\begin{remark}\label{EscalatoryNoIntegral}\emph{(proof in \cite[\S 4.10]{Hambleton})} 

For $D < 0$, if $r_{3}(D) = r_{3}(D') + 1$, i.e. if we are in the escalatory case, then there are no cubic fields of discriminant $3^{4}D$.

For $D > 0$, if $r_{3}(D) = r_{3}(D')$, i.e. if we are in the non-escalatory case, then there are no cubic fields of discriminant $3^{4}D$.
 \end{remark}
 
Remark~\ref{EscalatoryNoIntegral} is of key importance for Proposition~\ref{NoIntegral} and specifically for the proof that the curves $E_{D'}$ can have no integral points. But first we need a lemma, and the definition of the Fundamental $3$-Descent Map.

\begin{lemma} \label{NotInIm} For any odd fundamental discriminant $D$ coprime to $3$ with mirror discriminant $D'=-3D$, we consider the elliptic curve $E_{D'}:y^{2}=x^{3}+16D'$ and its $\phi$-isogenous $E_{D}: Y^{2}=X^{3}+16\cdot3^{4}D$. If there is an \emph{integral} point $P \in E_{D'}(\mathbb{Z})$, then this point cannot be the image $\hat{\phi}(Q)$ of any point $Q \in E_{D}(\mathbb{Q})$. \end{lemma} 
\begin{proof} Every point $Q \in E_{D}(\mathbb{Q})$ is of the form $Q=(\frac{X}{Z^{2}},\frac{Y}{Z^{3}})$ with $\text{gcd}(X,Z) = \text{gcd}(Y,Z) = 1$ \cite[pp.71-72]{SilvTate}. Let $d \coloneqq \text{gcd}(X,Y)$ and write $X=dX'$ and $Y=dY'$. Obviously $\text{gcd}(X',Y')=1$. As $Q \in E_{D}(\mathbb{Q})$ we obtain
\begin{equation}\label{eqn:PntDiv}
d^{2}Y'^{2}=d^{3}X'^{3}+2^{4}3^{4}Z^{6}D.
\end{equation}
As $D$ is an odd fundamental discriminant coprime to $3$ and $Z$ is coprime to $XY$, if any prime $p \neq 2,3$ divides $d$, then $p^{2} | D$ a contradiction. Hence $d | 36$. 

If $3 | d$ then equation~(\ref{eqn:PntDiv}) implies that $3^{2}||d$, $3 \nmid Y'$, and therefore $3^{2} || Y$, since $X'$ and $Y'$ are coprime and $3 \nmid 2^{4}Z^{6}D$.

Similarly, if $2 | d$ then $2^{2}||d$, $2 \nmid Y'$, and therefore $2^{2} || Y$, since $X'$ and $Y'$ are coprime and $2 \nmid 3^{4}Z^{6}D$.

Let us assume by way of contradiction that $\hat{\phi}(Q) = P = (A , B) \in E_{D'}(\mathbb{Z})$, for some $Q = (\frac{X}{Z^{2}} , \frac{Y}{Z^{3}}) \in E_{D}(\mathbb{Q})$, where $\hat{\phi}$ is given in equation~(\ref{themapdual}). The $y$-coordinate $(\hat{\phi}(Q))_{y} = B \in \mathbb{Z}$ is as follows:
$$(\hat{\phi}(Q))_{y} = \frac{Y}{27X^{3}Z^{3}}(X^{3}-2^{7}3^{4}DZ^{6}) = B \in \mathbb{Z}$$
$$\iff$$
$$YX^{3}-27X^{3}Z^{3}B = 2^{7}3^{4}Z^{6}DY.$$
Therefore, $3^{3} | YX^{3}$. From the equation of $E_{D}$, if either $X$ or $Y$ is divisible by $3$ then they both are. Hence, as we saw above, we should have that $3^{2} || d$ and $3^{2} || Y$. But the previous equation now implies that $3^{8} |  2^{7}3^{4}Z^{6}DY$, a contradiction since $3^{6} || 2^{7}3^{4}Z^{6}DY$.
\end{proof}

Let us come back to the discriminants $D$ and $D'$. There is a direct relation between the $3$-virtual units of $\mathcal{O}_{D'} \setminus \mathcal{O}_{D'}^{3}$ and the group of rational points of $E_{D'}$ established via the Fundamental $3$-Descent Map (e.g. \cite[Section 8.4.2]{Cohen}). Specifically to our curves $E_{D'}$, the Fundamental $3$-descent map $\Psi$  is defined as follows
\begin{equation}\label{DescentMap} \Psi: E_{D'}(\mathbb{Q}) \rightarrow Ker(N_{K_{D'}} : K_{D'}^{\times}/ {K_{D'}^{\times}}^{3} \rightarrow  \mathbb{Q}^{\times}/ {\mathbb{Q}^{\times}}^{3}),\end{equation}
$$(Q_{x} , Q_{y}) \mapsto (Q_{y} + 4\sqrt{D'}) \cdot {K_{D'}^{\times}}^{3}.$$
This map $\Psi$ is a group homomorphism with kernel equal to $\hat{\phi}(E_{D}(\mathbb{Q}))$ (e.g. \cite[Proposition 8.4.8]{Cohen}. 

\begin{remark}\label{rmrk:ImageNotCube}
Lemma~\ref{NotInIm} implies that the image of any integral point of $E_{D'}(\mathbb{Q})$ under the Fundamental $3$-Descent Map $\Psi$, is not a $K_{D'}^{\times}$-rational cube since it does not belong to the image of $\hat{\phi}$ (and therefore its image under $\Psi$ does not vanish). Recall that the cures $E_{D}$ and $E_{D'}$ do not have any non-trivial $\mathbb{Q}$-rational torsion points.
\end{remark}

\begin{proposition} \label{NoIntegral} We let $D$ be any odd fundamental discriminant satisfying $D \equiv 2 \bmod 3$, and we consider the elliptic curves $E_{D'}: y^{2}=x^{3}+16D'$. 

For $D<0$ whose corresponding quadratic number fields $K_{D}$ and $K_{D'}$ satisfy $r_{3}(D) = r_{3}(D') + 1$, the elliptic curves $E_{D'}$ have no integral points. 

For $D>0$ whose corresponding quadratic number fields $K_{D}$ and $K_{D'}$ satisfy $r_{3}(D) = r_{3}(D')$, the elliptic curves $E_{D'}$ have no integral points.
\end{proposition}
\begin{proof} We assume that we are in the escalatory case for $D < 0$ (respectively, that we are in the non-escalatory case for $D > 0$). Let us assume by way of contradiction that an elliptic curve $E_{D'}$ in this subfamily has an \emph{integral} point $P = (A , B) \in E_{D'}(\mathbb{Z})$. Plugging the point back in $E_{D'}$ and modifying the coefficients a little we obtain
$$3^{4}D = 4(\frac{3A}{4})^{3} - 27(\frac{B}{4})^{2}.$$

Case~(a): From the equation of $E_{D'}$, we see that if either $A$ or $B$ is even then so must be the other one, since $B^{2} = A^{3} - 3\cdot16D$. From the same equation we see that $B$ must actually be divisible by $4$ exactly, since $2 \nmid 3D$. Let $A=2a'$ and $B=4b$. Cancelling out the $16$ we end up with $2a'^{3} = b^{2} + 3D$. Since $D\equiv 1 \bmod 4$ and $b$ is odd, we have a contradiction unless $2 | a'$. Therefore, if both $A$ and $B$ are even then we must have that $4 ||B$ and $4 | A$, and we can write $A=2a'=4a$. This implies that the monic polynomial 
$$g(x) = x^{3} - \frac{3A}{4}x + \frac{B}{4} = x^{3} - 3ax +b$$ 
is in $\mathbb{Z}[x]$ and is of discriminant exactly $3^{4}D$. 

Consider the element $\lambda = \frac{\frac{B}{4}+\sqrt{D'}}{2} = \frac{b+\sqrt{D'}}{2} \in \mathcal{O}_{D'}$. We see that 
$$\lambda \bar{\lambda} = (\frac{A}{4})^{3} = a^{3}$$ and 
$$\lambda + \bar{\lambda} = \frac{B}{4} = b.$$ 
Given the polynomial $g(x) \in \mathbb{Z}(x)$ above, \cite[Proposition 4.1(1)]{Hambleton} implies that $\lambda$ is a $3$-virtual unit.

Denote by $\Lambda$ the element of $K_{D'}^{\times}$:
$$\Lambda = 2^{3}\lambda = \frac{2B + 8\sqrt{D'}}{2} = B + 4\sqrt{D'} \in K_{D'}^{\times}.$$ 
We recognise that $\Lambda$ is the image of the integral point $P=(A,B)$ under the Fundamental $3$-Descent Map $\Psi$ that we defined in (\ref{DescentMap}). Since by \cite[Proposition 8.4.8]{Cohen} 
the kernel of $\Psi$ is precisely $\hat{\phi}(E_{D}(\mathbb{Q}))$, then by Lemma~\ref{NotInIm} and Remark~\ref{rmrk:ImageNotCube}, $\Psi(P)$ does not vanish and $\Psi(P) = \Lambda \equiv \lambda \in K_{D'}^{\times}/(K_{D'}^{\times})^{3}$. Then, by \cite[Proposition 4.1 (2)]{Hambleton}, $g(x)$ is irreducible over $\mathbb{Q}$. Finally, since $\lambda$ is a primitive $3$-virtual unit, \cite[Theorem 4.4]{Hambleton} implies that $g(x)$ generates a cubic field of discriminant $D$ or $3^{4}D$. To decide which one, we need to show that the coefficients $3a$ and $b$ of $g(x)$ satisfy certain congruence conditions:
\begin{enumerate}[label=(\roman*)]
\item\label{one} If $3$ divides either  $A$ or $B$, then $9 | 16D'$ which is impossible. Furthermore, $b$ is odd and $D$ is squarefree. Therefore, we easily deduce that $g(x)$ is in standard form with respect to all primes.
\item\label{two} Observe that the equivalence 
$$6 \equiv 3D = 4a^{3}-b^{2} \bmod 9$$ implies immediately that $a \not \equiv 2 \bmod 3$. Hence 
$$a \equiv 1 \bmod 3 \ \ \text{and therefore} \ \ 3a \equiv 3 \bmod 9.$$
\item\label{three} We observe that $6 \equiv 4a^{3}-b^{2} \bmod 9 \iff b^{2} \equiv 4a^{3} - 6 \bmod 9$. If now we assume, by way of contradiction, that $b^{2} \equiv 4 \bmod 9$ then the above equivalence gives $4a^{3} - 10 \equiv 0 \bmod 9$, which is impossible. Therefore we must have that $b^{2} \not \equiv 4 \bmod 9$.
\end{enumerate}
Given \ref{one}, \ref{two} and \ref{three} of Case~(a) above, we deduce from \cite[Theorem 2.14]{Hambleton} that $g(x)$ generates a cubic field of discriminant $3^{4}D$. And now, this leads to a contradiction since we assumed that we are in the escalatory case for $D < 0$ (respectively, non-escalatory case for $D > 0$). 

Case~(b): Both $A$ and $B$ are odd. Then $g(x) \notin \mathbb{Z}[x]$ but the following polynomial $f(x)$ does have integer coefficients and it is of discriminant $2^{6}3^{4}D$: 
$$f(x) = x^3 - 3Ax + 2B \in \mathbb{Z}[x].$$ 
As above, let $\lambda = \frac{B/4 + \sqrt{D'}}{2}$. Now $\lambda$ is not integral but we observe that
 $$\Lambda = 2^{3}\lambda = \frac{2B + 8\sqrt{D'}}{2} = B + 4\sqrt{D'} = \Psi(P) \in K_{D'}^{\times}/(K_{D'}^{\times})^{3},$$ and $\Psi(P)$ does not vanish by Lemma~\ref{NotInIm} and Remark~\ref{rmrk:ImageNotCube}. Since $\Lambda \bar{\Lambda} = A^{3}$ and $\Lambda + \bar{\Lambda} = 2B$, by \cite[Proposition 4.1, (1) and (2)]{Hambleton}, $\Lambda$ is a $3$-virtual unit and $f(x)$ is irreducible. Furthermore, since $B$ is odd, $\Lambda$ is a primitive $3$-virtual unit and again, by \cite[Theorem 4.4]{Hambleton}, $f(x)$ generates a cubic field of discriminant $D$ or $3^{4}D$. As above, we need to examine whether the same congruence conditions hold, this time for the coefficients $3A$ and $2B$ of the polynomial $f(x)$:
 \begin{enumerate}[label=(\roman*)]
 \item\label{bone} As $gcd(6,AB) = 1$ and $D$ is squarefree, the polynomial $f(x)$ is in standard form with respect to all primes.
 \item\label{btwo} We observe that $A^{3} - B^{2} = 3\cdot 16 D \equiv 6 \bmod 9$. Therefore, the same rational as in Case~(a)\ref{two} above forces $A$ to be equivalent to $1$ modulo $3$, which in turn gives $3A \equiv 3 \bmod 9$.
 \item\label{bthree} Regarding the constant coefficient $2B$ of $f(x)$, we see that in this Case~b we have $$A^{3} - B^{2} = 3 \cdot 16 D \equiv 6 \bmod 9 \iff  B^{2} \equiv A^{3} - 6 \equiv 0 \bmod 9$$
 $$\iff (2B)^{2} \equiv 4A^{3} - 6 \bmod 9.$$ 
But as in Case~(a)\ref{three} above, we see that $(2B)^{2} \not \equiv 4 \bmod 9$ since $4A^{3} \equiv 10 \bmod 9$ has no solutions.
 \end{enumerate}
 Given \ref{one}, \ref{two} and \ref{three} of Case~(b) above, we deduce from \cite[Theorem 2.14]{Hambleton} that $f(x)$ generates a cubic field of discriminant $3^{4}D$. And now, this leads to a contradiction since we assumed that we are in the escalatory case for $D < 0$ (respectively, non-escalatory case for $D > 0$). 
 \end{proof}

 \begin{corollary}\label{cor:oddranknoint}
 Assuming finiteness of the $3$-primary part of the Tate-Shafarevich group, for every odd and negative fundamental discriminant $D \equiv 2 \bmod 3$ whose corresponding quadratic number fields $K_{D}$ and $K_{D'}$ satisfy $r_{3}(D) = r_{3}(D') + 1$, the elliptic curves $E_{D'}: y^{2}=x^{3}+16D'$ have odd rank and no integral points. \end{corollary}
 \begin{proof}
Immediate from Corollary~\ref{oddrank} and Proposition~\ref{NoIntegral}.
 \end{proof}

\section{A parametrised family of elliptic curves $E_{D'}$ with odd rank and no integral points}\label{Example}
For any $w \in \mathbb{N}$, with $w \equiv 1 \bmod 4$ and $w \equiv 0 \bmod 3$, let us define the negative integer
$$D(w) = -(2^{4}3^{7}w^{2} + (2+2^{4}3^{7})w + 2^{2}3^{7}+1) = - f(w).$$
We observe that 
$D(w) \equiv 1 \bmod 4$ and $D(w) \equiv 2 \bmod 3$.

The restrictions on $w$ imply that $w \equiv 9 \bmod 12$. Let us write $w = 9 + 12k$, for $k \in \mathbb{N}$. The quadratic polynomial $f(w) = f(9+12k) \eqqcolon F(k)$ becomes
$$F(k) = 2^{8}3^{9}k^{2} + 2^{3}\cdot3(1 + 2^{3}3^{7} + 2^{4}3^{9})k + 2\cdot3^{2}(1 + 2\cdot3^{5} + 2^{3}3^{8} + 2^{3}3^{9}) + 1.$$ 

\begin{lemma}
With notation as above, the following set of odd negative fundamental discriminants is infinite:
$$\mathcal{D}_{w} = \{ D(w) \ | \ w \equiv 9 \bmod 12 \ \text{and} \ D(w) \ \text{squarefree} \}.$$   
\end{lemma}
\begin{proof}
The quadratic polynomial $F(k)$ has a positive leading coefficient and it is straightforward to infer that $F(k) \neq (ak+b)^{2}$ for any $a,b \in \mathbb{Z}$. Therefore, by Erd\"os's remark in \cite[pp.1-2]{Erdos}, there are infinitely many $k \in \mathbb{N}$ for which $F(k)$ is squarefree. As a result, there are infinitely many $w \in \mathbb{N}$ for which $f(w)$ is squarefree and therefore $\mathcal{D}_{w}$ is an infinite parametrised set of odd negative fundamental discriminants satisfying $D(w) \equiv 2 \bmod 3$.
\end{proof}

It is a straightforward computation to see that the point 
$$P(w) = (1/9 , y(w)/27), \ \text{where} \ y(w) = 2^{4}3^{7}w+2^{3}3^{7}+1,$$ belongs to 
$E_{D(w)'}(\mathbb{Q})$ for any integer $w$. When additionally $D(w) \in \mathcal{D}(w)$, we show below that this point belongs to the set $E_{D(w)'}(\mathbb{Q}) \setminus \hat{\phi}(E_{D(w)}(\mathbb{Q}))$, which implies in particular that $r(E_{D(w)'}) \geq 1$ unconditionally.

\begin{proposition}\label{prop:InfOrder}
For every $D(w) \in \mathcal{D}(w)$, consider the point $$P(w) = (1/9 , y(w)/27), \ \text{with} \ y(w) = 2^{4}3^{7}w+2^{3}3^{7}+1.$$
The point $P(w)$ belongs to the set $E_{D(w)'}(\mathbb{Q}) \setminus \hat{\phi}(E_{D(w)}(\mathbb{Q}))$, and therefore the corresponding elliptic curve $E_{D(w)'}/\mathbb{Q}$ has positive rank.
\end{proposition}
\begin{proof}
It is a straightforward computation to see that $P(w) \in E_{D(w)'}(\mathbb{Q})$, where as above $E_{D(w)'}: y^{2} = x^{3} - 3\cdot 16D(w).$
Now assume by way of contradiction that $P(w) = \hat{\phi}(Q)$, for some $Q=(X/Z^{2},Y/Z^{3}) \in E_{D(w)}(\mathbb{Q})$. Then, equation~\ref{themapdual} implies that 
$$1/9 = \frac{(X/Z^{2})^{3} - 2^{6}3^{3}(-3D(w))}{9(X/Z^{2})^{2}} \iff Z^{2}X^{2} = X^{3} + 2^{6}3^{4}D(w)Z^{6}.$$
Since $gcd(X,Z) = 1$ this implies that $Z=1$ and $X^{2}(X-1) = 2^{6}3^{4}|D(w)|$.
Now, as $X^{2}$ must divide the right hand side, since $gcd(6,D(w))=1$ and $D(w)$ is squarefree, this implies that $1 < X \ | \ 2^{3}3^{2}$.

If $2 \nmid X$, since $X > 1$ we must have $3 | X$ and hence $3 \nmid X - 1$, in which case $X = 3^{2}$ and $X-1 = 2^{6}|D(w)| < X = 9$, a contradiction.

Similarly, if $3 \nmid X$, since $X > 1$ we must have $2 | X$ and hence $2 \nmid X - 1$, in which case $X = 2^{3}$ and $X-1 = 3^{4}|D(w)| < X = 8$, again a contradiction.  

Finally, if both $2$ and $3$ divide $X$ then neither can divide $X-1$ and we have $X = 2^{3}3^{4}$ and $X-1 = |D(w)| = 2^{3}3^{4} - 1 = 71$. But this would imply that $D(w) = -71$ which is a contradiction because, even though $-71$ is a fundamental discriminant, it is not equivalent to $2 \bmod 3$.

We have shown that $P(w) \neq \hat{\phi}(Q)$ for any $Q \in E_{D(w)}(\mathbb{Q})$. Equation~\ref{R} now implies that $1 \leq \text{dim}_{\mathbb{F}_{3}}(E_{D(w)'}(\mathbb{Q})/\hat{\phi}(E_{D(w)}(\mathbb{Q})) \leq r(E_{D(w)'})$.
\end{proof}

\begin{corollary}\label{cor:uncond}
Let $\mathcal{D}_{w}^{esc} = \{ D(w) \ | \ r_{3}(D(w)) = r_{3}(D(w)') + 1 \}$. For every $D(w) \in \mathcal{D}_{w}^{esc}$, the corresponding elliptic curve $E_{D(w)'}$ has infinitely many rational points but no integral points.
\end{corollary}
\begin{proof} Immediate from Propositions~\ref{NoIntegral}~\&~\ref{prop:InfOrder}.
\end{proof}

Finally, as also discussed in the Introduction, the set $\mathcal{D}_{w}^{esc}$ is not empty. We performed a straightforward search in PARI/GP for $1 \leq w \leq 10^{5}$ with $w \equiv 9 \bmod 12$ and $D(w)$ squarefree. The output was $4,842$ such discriminants $D(w)$ that lie within the bounds
$$D(9) = -3158047 \leq D(w) \leq D(99993) = -349874512078399.$$
According to Corollary~\ref{cor:uncond}, these discriminants yield $4,842$ corresponding elliptic curves $E_{D(w)'}$ with non-trivial rank and no integral points.

\section*{Acknowledgements} The author would like to thank Antoine Joux and the CISPA Helmholtz Center for Information Security, Pieter Moree and the Max Planck Institute for Mathematics in Bonn, and Karl Jansen and the Deutsches Elektronen-Synchrotron DESY-Zeuthen, for their hospitality and support while working on this paper.

\begin{figure}[!b]
   \includegraphics[width = 0.12\textwidth]{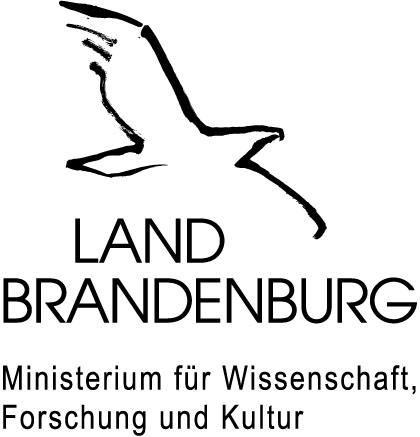}
\end{figure}

\end{document}